\documentclass[12pt]{amsart}
\usepackage{color}
\usepackage[all]{xy}
\usepackage{amssymb,amsmath}

\usepackage{setspace}

\calclayout
\newtheorem{dummy}{anything}[section]
\newtheorem{theorem}[dummy]{Theorem}

\newtheorem{lemma}[dummy]{Lemma}
\newtheorem{proposition}[dummy]{Proposition}
\newtheorem{corollary}[dummy]{Corollary}

\theoremstyle{definition}
\newtheorem{definition}[dummy]{Definition}
  \newtheorem{example}[dummy]{Example}
  
  \newtheorem{remark}[dummy]{Remark}

\newcommand{\cC}{\mathcal C}

\newcommand{\cE}{\mathcal E}
\newcommand{\cF}{\mathcal F}

\newcommand{\cU}{\mathcal U}
\newcommand{\cV}{\mathcal V}
\newcommand{\cW}{\mathcal W}

\newcommand{\bZ}{\mathbf Z}

\newcommand{\bbF}{\mathbb F}

\newcommand{\bbZ}{\mathbb Z}

\DeclareMathOperator{\Hom}{Hom} 

\DeclareMathOperator{\rk}{rk}
\DeclareMathOperator{\image}{im}
\DeclareMathOperator{\Ext}{Ext}

\DeclareMathOperator{\Ind}{Ind}
\DeclareMathOperator{\Res}{Res} 
\DeclareMathOperator{\Mor}{Mor}

\newcommand{\la}{\langle}
\newcommand{\ra}{\rangle}
\newcommand{\lra}{\longrightarrow}

\newcommand{\bd}{\partial}
\newcommand{\id}{\mathrm{id}}

\def\G{\varGamma}

\def\ep{\varepsilon}

\newcommand{\leftexp}[2]{{\vphantom{#2}}^{ #1}{\hskip-1pt#2}}

\DeclareMathOperator{\Or}{Or}

\newcommand\RG{R\G}

\newcommand\OrG{\Or _{\cF}G}
\newcommand\OrH{\Or_{\cF}H}

\newcommand\un{\underline}
\newcommand\uR[1]{R[{#1}^{\, \textbf{?}\,}]}

\newcommand{\ot}{\otimes}
\def\maprt#1{\smash{\,\mathop{\longrightarrow}\limits^{#1}\,}}

\begin{document}

\title{Relative Group Cohomology and The Orbit Category}
\author{Semra Pamuk}
\author{Erg\"un Yal\c c\i n}

\address{Department of Mathematics, Middle East Technical University, 06531,Ankara, Turkey }

\email{pasemra@metu.edu.tr }

\address{Department of Mathematics, Bilkent University,
06800 Bilkent, Ankara, Turkey}

\email{yalcine@fen.bilkent.edu.tr }

\date{\today}

\thanks{2010 {\it Mathematics Subject Classification.} Primary: 20J06; Secondary: 55N25.} \thanks{The second author is partially supported by T\" UB\. ITAK-TBAG/110T712 and by T\" UB\. ITAK-B\. IDEB/2221 Visiting Scientist Program.}

\begin{abstract} Let $G$ be a finite group and $\cF$ be a family of subgroups of $G$ closed under conjugation and taking subgroups. We consider the question whether there exists a periodic relative $\cF$-projective resolution for $\bbZ$ when $\cF$ is the family of all subgroups $H \leq G$ with $\rk H \leq \rk G-1$. We answer this question negatively by calculating the relative group cohomology $\cF H^* (G, \bbF_2)$ where $G=\bbZ/2\times \bbZ /2$ and $\cF$ is the family of cyclic subgroups of $G$. To do this calculation we first observe that the relative group cohomology $\cF H^*(G, M)$ can be calculated using the ext-groups over the orbit category of $G$ restricted to the family $\cF$. In second part of the paper, we discuss the construction of a spectral sequence that converges to the cohomology of a group $G$ and whose horizontal line at $E_2$ page is isomorphic to the relative group cohomology of $G$.
\end{abstract}

\maketitle

\section{Introduction}
\label{sect:intro} 
Let $G$ be a finite group and $R$ be a
commutative ring of coefficients. For every $n \geq 0$, the $n$-th cohomology group $H^n (G, M)$ of $G$ with coefficients in an $RG$-module $M$ is defined as the $n$-th cohomology group of the cochain complex $\Hom _{RG} (P_* , M)$ where $P_*$ is a projective resolution of $R$ as an $RG$-module. Given a family $\cF$ of subgroups of $G$ which is closed under conjugation and taking subgroups, one defines the relative group cohomology $\cF H^*(G, M)$ with respect to the family $\cF$ by adjusting the definition  in the following way: We say a short exact sequence of $RG$-modules is $\cF$-split if it splits after restricting it to the subgroups $H$ in $\cF$. The definition of projective resolutions is changed accordingly using $\cF$-split sequences (see Definition \ref{def:F-projective resolution}). Then, for every $RG$-module $M$,
the relative group cohomology $\cF H^*(G,M)$ with respect to the family $\cF$ is defined as the cohomology of the cochain complex $\Hom _{RG}(P_*, M)$ where 
$$P_*:\ \cdots \to P_n \maprt{\partial _n} P_{n-1} \to \cdots \to P_0 \to R\to 0$$ is a relative $\cF$-projective resolution of $R$.  

Computing the relative group cohomology is in general a difficult task. Our first theorem gives a method for computing relative group cohomology using ext-groups over the orbit category. In general calculating ext-groups over the orbit category is easier since there are many short exact sequences of modules over the orbit category which come from the natural filtration of the poset of subgroups in $\cF$. To state our theorem, we first introduce some basic definitions about orbit categories. 

The orbit category $\G=\OrG$ of the group $G$ with respect to the family $\cF$ is defined as the category whose objects are orbits of the form $G/H$ where $H \in \cF$ and whose morphisms from $G/H$ to $G/K$ are given by $G$-maps from $G/H$ to $G/K$. An $R\G$-module is defined as a contravariant functor from $\G$ to the category of $R$-modules. We often denote the $R$-module $M(G/H)$ simply by $M(H)$ and call $M(H)$ the value of $M$ at $H\in \cF$. The maps $M(H) \to M(K)$ between two subgroups $H$ and $K$ can be expressed as compositions of conjugations and restriction maps. The category of $R \G$-modules has enough projectives and injectives, so one can define ext-groups for a pair of $R\G$-modules in the usual way.

There are two $R\G$-modules which have some special importance for us. The first one is the constant functor $\underline R$ which has the value $R$ at $H$ for every $H\in \cF$ and the identity map as maps between them. The second module that we are interested in is the module $M^?$ which is defined for any $RG$-module $M$ as the $R\G$-module that takes the value $M^H$ at every $H \in \cF$ with the usual restriction and conjugation maps coming from the restriction and conjugation of invariant subspaces. Our main computational tool is the following:

\begin{theorem}\label{thm:main 1; computational tool} Let $G$ be a finite
group and $\cF$ be a family of subgroups of $G$ closed under conjugation and taking subgroups. Then, for every $RG$-module $M$,  $$ \cF H^* (G, M)\cong \Ext _{R\G} ^* (\underline{R}, M^? ).$$
\end{theorem}

This theorem allows us to do some computations which have some
importance for the construction of finite group actions on spheres. One of the ideas for constructing group actions on spheres is to construct chain complexes of finitely generated permutation modules of certain isotropy type and then find a $G$-$CW$-complex which realizes this permutation complex as its chain complex. One of the questions that was raised in this process is the following: Given a
finite group $G$ with rank $r$, if we take $\cF$ as the family of all subgroups $H$ of $G$ with $\rk H \leq r-1$, then does there exist an $\cF$-split sequence
of finitely generated permutation modules $\bbZ X_i$ with isotropy in $\cF$ such that
$$ 0 \to \bbZ \to \bbZ X_n \to \cdots \to \bbZ X_2 \to \bbZ X_1 \to \bbZ X_0 \to \bbZ \to 0$$
is exact? We answer this question negatively by calculating the relative group cohomology of the Klein four group relative to its cyclic subgroups. Note that if there were an exact sequence as above, then by splicing it with itself infinitely many times we could obtain a relative $\cF$-projective resolution and as a consequence the relative group cohomology $\cF H^* (G, \bbF _2)$ would be periodic. We prove that this is not the case.

\begin{theorem}\label{thm:main 2; not periodic}
Let $G=\bZ /2 \times \bZ /2$ and $\cF$ be the family
of all cyclic subgroups of $G$. Then, $\cF H^* (G, \bbF _2)$ is not
periodic.
\end{theorem}

The proof of this theorem is given by computing the
dimensions of $\cF H^i (G, \bbF _2)$ for all $i$ and showing that the dimensions grow by the sequence $(1, 0, 1, 3, 5, 7, \dots )$. In the
computation, we use Theorem \ref{thm:main 1; computational tool} and some short exact sequences coming from the poset of subgroups of $G$.  

In the rest of the paper, we discuss the connections between relative group cohomology and higher limits. Given two families $\cU \subseteq \cV$ of subgroups of $G$, the inverse limit functor $$\lim _{\longleftarrow}{} _{\cU} ^{\cV} : R \G _{\cU} \to R \G _{\cV},$$ where $\G _{\cU}=\Or _{\cU}G$ and $\G _{\cV} =\Or _{\cV} G$, is defined as the functor which is right adjoint to the restriction functor (see Definition \ref{def:relativelimit} and Proposition \ref{pro:adjointness}). The limit functor is left exact, so the $n$-th higher limit $(\lim_{\cU} ^{\cV})^n$ is defined as the $n$-th right derived functor of the limit functor. Compositions of limit functors satisfy the identity $$ {\rm lim} _{\cU} ^{\cW} ={\rm lim} _{\cV} ^{\cW} \circ {\rm lim} _{\cU} ^{\cV}\ .$$ So  there is a Grothendieck spectral sequence for the right derived functors of the limit functor. A special case of this spectral sequence gives a spectral sequence that converges to the cohomology of a group and whose horizontal line is isomorphic to the relative group cohomology.

\begin{theorem}[Theorem 6.1, \cite{martinez}]\label{thm:main 3; spect. seq.}
Let $G$ be a finite group and $R$ be a commutative coefficient ring. Let $\G=\OrG$ where $\cF$ is a family of subgroups of $G$ closed under conjugation and taking subgroups. Then, for every
$RG$-module $M$, there is a first quadrant spectral sequence
$$ E_2^{p,q}=\Ext_{R\G} ^p (\underline R , H^q (? , M) )\Rightarrow H^{p+q}(G,M).$$
In particular, on the horizontal line, we have $E_2 ^{p,0}\cong \cF
H^p (G, M )$.
\end{theorem}

This is a special case of a spectral sequence constructed 
by Mart{\' i}nez-P{\' e}rez \cite{martinez} and it is stated as a theorem (Theorem 6.1) in \cite{martinez}. There is also a version of this sequence for infinite groups constructed by Kropholler \cite{kropholler} using a different approach. In Section \ref{sect:higherlimits}, we discuss the edge homomorphisms of this spectral sequence and the importance of this spectral sequence for approaching the questions related to the essential cohomology of finite groups. We also discuss how this spectral sequence behaves in the case where $G=\bbZ /2 \times \bbZ /2$ and $\cF$ is the family of cyclic subgroups of $G$. 

The paper is organized as follows:  In Section \ref{sect:relativecohomology}, we review the concepts of $\cF$-split sequences and relative projectivity of an $RG$-module with respect to a family of subgroups $\cF$ and define relative group cohomology $\cF H^* (G, M)$. In Section \ref{sect:orbitcategory}, orbit category and ext-groups over the orbit category are defined and Theorem \ref{thm:main 1; computational tool} is proved. Then in Section \ref{sect:computation}, we perform some computations with the ext-groups over the orbit category and prove Theorem \ref{thm:main 2; not periodic}. In Sections \ref{sect:limitfunctor} and \ref{sect:higherlimits}, we introduce the higher limits and construct the spectral sequence stated in Theorem \ref{thm:main 3; spect. seq.}.

\section{Relative group cohomology}
\label{sect:relativecohomology}

Let $G$ be a finite group, $R$ be a commutative ring of
coefficients, and $M$ be a finitely generated $RG$-module. 
In this section we introduce the definition of relative group
cohomology $\cF H^* (G, M)$ with respect to a family of subgroups $\cF$. When we say $\cF$ is a family of subgroups of $G$, we always mean that $\cF$ is closed under conjugation and taking subgroups, i.e., if $H\in\cF$ and $K\leq G$ such that $K^g \leq H $, then $K \in \cF$.  

\begin{definition}\label{def:F-split}
A short exact sequence $\cE : 0 \to A \to B \to C \to 0$ of $RG$-modules is called $\cF$-\emph{split} if for every $H \in \cF$, the restriction of $\cE$ to $H$ splits as an extension of  $RH$-modules.
\end{definition}

For a $G$-set $X$, there is a notion of $X$-split sequence defined as follows:

\begin{definition}\label{def:X-split}
Let $X$ be a $G$-set and let $RX$ denote the permutation module with the basis given by $X$.  Then, a short exact sequence $0 \to A \to B \to C \to 0$ of $RG$-modules is called $X$-split if the sequence
\begin{equation*}
0\to A\otimes_{R}RX \to B\otimes_{R} RX \to C\otimes_{R }RX \to 0\end{equation*} 
splits as a sequence of $RG$-modules.
\end{definition}

These two notions are connected in the following way:

\begin{proposition}[Lemma 2.6, \cite{nucinkis1}]\label{pro:F-split vs X-split} Let $G$ be a finite
group and $\cF$ be a family of subgroups of $G$. Let $X$ be a
$G$-set such that $X^H\neq\emptyset$ if and only if $H\in\cF.$
Then, a sequence $0 \to A \to B \to C \to 0$ of $RG$-modules is $\cF$-split if and only if it is $X$-split.
\end{proposition}

\begin{proof} We first show  that given a short exact sequence $ 0 \to A \maprt{i} B \maprt{\pi} C\to 0$ of $RG$-modules, its restriction to $H\leq G$ splits as a sequence of $RH$-modules if and only if the sequence
\begin{equation}\label{eqn:tensor with R[G/H]}
\xymatrix{0\ar[r] & A\otimes_R R[G/H] \ar[r]^{i \ot \id} & B\otimes_R R[G/H] \ar[r]^{\pi \ot \id} &  C\otimes_R R[G/H]\ar[r] & 0}
\end{equation}
splits as a sequence of $RG$-modules. Since $A\otimes_{R}
R[G/H]\cong \Ind _H ^G \Res ^G _H A$, the ``only if" direction is clear. For the ``if" direction assume that the sequence
(\ref{eqn:tensor with R[G/H]}) splits. Let $s$ be a splitting for $\pi \ot \id$. Then consider the following diagram
$$\xymatrix{B \ar[r]^{\pi} & C \ar@<1ex> [d]^{\eta}\\
B\otimes_RR[G/H]\ar[u]^{\id\otimes\ep}\ar[r]_{\pi \otimes \id}&C \otimes_RR[G/H]\ar@/_/[l]_{s}\ar@<1ex>[u]^{\id\otimes\ep}}$$ where $\ep$ is the augmentation map $\ep: R[G/H]\to R$ which takes $gH $ to $1\in R$ for all $g \in G$ and $\eta$ is the map defined by $\eta (c)=c\ot H$. Define $\hat s : C \to B$ to be the composition $(\id \ot \ep) s \eta$. Then we have
$$ \pi \hat{s} = \pi (\id\otimes\ep ) s \eta = (\id\otimes\ep )(\pi \otimes \id)s \eta =(\id\otimes\ep )\eta =\id.$$ Since $\eta$ is an $H$-map, the splitting $\hat s$ is also an $H$-map. Thus, the short exact sequence $0 \to A \to B \to C \to 0$ splits when it is restricted to $H$.

Now, the general case follows easily since $RX\cong \oplus _{i\in I} R[G/H_i]$ for a set of subgroups $H_i \in \cF$ satisfying the following condition: if $H \in \cF $, then $H^g \leq H_i$ for some $g\in G$ and $i \in I$.
\end{proof}

Now, we define the concept of relative projectivity.

\begin{definition}\label{def:F-projective}
An $RG$-module $P$ is called $\cF$-{\it projective} if for every $\cF$-split sequence of $RG$-modules $0\to A \to B \to C \to 0$ and an $RG$-module map $\alpha: P \to C$, there is an $RG$-module map $\beta : P \to B$ such that the following diagram commutes
$$\xymatrix{&&&P\ar[d]^-{\alpha}\ar@{-->}[dl]_-{\beta} \\
0\ar[r]&A\ar[r]&B\ar[r]^{\pi}&C\ar[r]&0\ .}$$
\end{definition}

Given a $G$-set $X$, we say $X$ is $\cF$-free if for every $x$ in $X$ the isotropy subgroup $G_x$ belongs to $\cF$. 
An $RG$-module $F$ is called an $\cF$-free module if it is isomorphic to a permutation module $RX$ where $X$ is an $\cF$-free $G$-set. Note that an $\cF$-free $RG$-module is isomorphic to a direct sum of the form $\oplus _i R[G/H_i]$ where $H_i \in \cF$ for all $i$.

\begin{proposition}\label{pro:proj is direct summand of free}
An $RG$-module $M$ is $\cF$-projective if and only if it is a direct summand of an $RG$-module of the form $N \otimes _R RX$ where $RX$ is an $\cF$-free module and $N$ is an $RG$-module.
\end{proposition}

\begin{proof} Let $X$ be a $G$-set with the the property that $X^H \neq \emptyset$ if and only if $H \in \cF$. Then the sequence $0 \to \ker \ep \maprt{} RX \maprt{\ep} R \to 0$ where $\ep (\sum a_x x)=\sum a_x$ is an $\cF$-split sequence since its restriction to any subgroup $H \in \cF$ splits. Tensoring this sequence with $M$, we get an $\cF$-split sequence $$ 0 \to M \otimes_R
\ker \ep \to M \otimes_R RX \to M \to 0.$$ If $M$ is $\cF$-projective, then this sequence splits and hence $M$ is a direct summand of $M \otimes_R RX$. For the converse, it is enough to show that an $RG$-module of the form $N \otimes_R RX$ is projective. If $RX=\oplus _i R[G/H_i]$, then $$N\otimes_R RX\cong \oplus _i \Ind_{H_i} ^G \Res ^G _{H_i} N.$$ So, we need to show that for every $H \in \cF$, an $RG$-module of the form $\Ind _{H_i} ^G \Res ^G _{H_i} N$ is $\cF$-projective. This follows from Frobenious reciprocity (see 
\cite[Corollary 2.4]{nucinkis1} for more details).
\end{proof}

Note that in the argument above, we have seen that for every
$RG$-module $M$, there is an $\cF$-split surjective map $ M \otimes _RRX \to M$ where $M \otimes_R RX$ is an $\cF$-projective module. Inductively taking such maps, we obtain a projective resolution of $M$ formed by $\cF$-projective modules. Note that each short exact sequence appearing in the construction is $\cF$-split. The resolutions that satisfy this property are given a special name.

\begin{definition}\label{def:F-projective resolution} Let $M$ be an $RG$-module.
A relative $\cF$-projective resolution $P_*$ of $M$ is
an exact sequence of the from $$ \cdots \to P_n \maprt{\bd _n }
P_{n-1} \to \cdots \to P_2 \maprt{\bd _2} P_1 \maprt{\bd _1} P_0 \maprt{\bd _0} M \to 0$$ where for
each $n\geq 0$, the $RG$-module $P_n$ is $\cF$-projective and the
short exact sequences $$0 \to \ker \bd _{n} \to P_n \to \image \bd _n
\to 0$$ are $\cF$-split.
\end{definition}

In \cite[Lemma 2.7]{nucinkis1}, it is shown that there is a version of Schanuel's lemma for $\cF$-split sequences. This follows from the fact that the class of $\cF$-split exact
sequences is proper. Note that the concept of relative projective resolution is the same as proper projective resolutions for the class of $\cF$-split exact sequences. Thus, we have the following:

\begin{proposition}\label{pro:uniqueness} Let $M$ be an $RG$-module. Then, any two relative $\cF$-projective resolutions of $M$ are chain homotopy equivalent.
\end{proposition}

We can now define the relative cohomology of a group as follows:

\begin{definition}\label{def:relative cohomology} Let $G$ be a finite group and $\cF$ be a family of subgroups
of $G$. For every $RG$-module $M$ and for each $n\geq 0$, the $n$-th relative cohomology of $G$ is defined as the cohomology group $$\cF H^n (G,M):=H^n (\Hom _{RG} (P_* , M)$$
where $P_*$ is a relative $\cF$-projective resolution of $R$.
\end{definition}

If $\cF$ is a collection of subgroups of $G$ which is not
necessarily closed under conjugation and taking subgroups, we can still define cohomology relative to this family in the following way. Let $\overline \cF$ be a family defined by $$\overline \cF =\{K \leq G \ | \ K^g \leq H \text{ for some } g\in G \text{ and } H\in\cF \}.$$ We call $\overline \cF$ the subgroup closure of $\cF$. Then, relative cohomology with respect to $\cF$ is defined in the following way:

\begin{definition}\label{def:relative cohomology2}
Let $G$ be a finite group and $\cF$ be a collection
of subgroups of $G$. For a $RG$-module $M$, the relative cohomology
of $G$ with respect to $\cF$ is defined by
$$\cF H^n (G, M):=\overline \cF H^n (G, M)$$
where $\overline \cF$ is the subgroup closure of $\cF$.
\end{definition}

This definition makes sense since a short exact sequence is $\cF$-split if and only if it is $\overline \cF$-split. So, the corresponding proper categories are equivalent. Note that when $\cF=\{ H\}$, the definition above coincides with the
definition of cohomology of a group relative to a subgroup $H$ (see \cite[Section 3.9]{benson1}).

\section{Ext-groups over the orbit category}
\label{sect:orbitcategory}

Let $G$ be a finite group and $\cF$ be a family of subgroups of $G$. As before, we assume that $\cF$ is closed under conjugation and taking subgroups. The orbit category $\Or_{\cF}(G)$ of $G$ relative to $\cF$ is defined as the category whose objects are orbits of the form $G/H$ with $H \in \cF$ and whose morphisms from $G/H$ to $G/K$ are given by set of $G$-maps $G/H\to G/K$. We denote the orbit category $\OrG$ by $\G$ to simplify the notation. In fact, for almost everything about orbit categories we follow the notation and terminology in \cite{lueck}.

Let $R$ be a commutative ring. An $\RG$-module is a
contravariant functor from $\G$ to the category of $R$-modules. An $\RG$-module $M$ is sometimes called a coefficient system and used in the definition of Bredon cohomology as coefficients. Since an $\RG$-module is a functor onto an abelian category, the category of $\RG$-modules is an abelian category and the usual tools for doing homological algebra are available. In particular, a sequence $M'\lra M\lra M''$ of $R\G$-modules is exact if and only if $$ M'(H)\lra M(H)\lra M''(H)$$ is an exact sequence of $R$-modules for every $H\in\cF$. The notions of submodule, quotient module, kernel, image, and cokernel are defined objectwise.  The direct sum of $R\G$-modules is given by taking the usual direct sum objectwise. The Hom functor has the following description.

\begin{definition}
Let $M,N$ be $R\G$-modules. Then,
$$\Hom_{R\G}(M,N)\subseteq\bigoplus_{H\in\cF}
\Hom_{R}(M(H),N(H))$$ is the $R$-submodule of morphisms $f_H: M(H) \to N(H)$ satisfying the relation $f_K\circ M(\varphi) =N(\varphi) \circ f_H$
for every morphism $\varphi:G/K\lra G/H$.
\end{definition}

Recall that by the usual definition of projective modules, an
$\RG$-module $P$ is projective if and only if the functor
$\Hom_{R\G}(P,-)$ is exact.

\begin{lemma} For each $K\in\cF$, let $P_K$ denote the
$\RG$-module defined by
$$P_K(G/H)=R\Mor(G/H,G/K)$$ where $R\Mor(G/H,G/K)$ is the free
abelian group on the set $\Mor(G/H,G/K)$ of all morphisms $G/H \to G/K$. Then, $P_K$ is a
projective $\RG$-module.
\end{lemma}

\begin{proof} It is easy to see that for each $\RG$-module $M$, we have
$\Hom_{R\G}(P_K,M)\cong M(K).$ Since the exactness is defined
objectwise, this means the functor $\Hom _{R\G} (P_K, -)$ is exact. Hence we can conclude that $P_K$ is projective.
\end{proof}

The projective module $P_K$ is also denoted by $\uR{G/K}$ since $P_K(G/H)\cong R[(G/K)^H]$. One often calls $\uR{G/K}$ a free $R\G$-module since all
the projective $\RG$-modules are summands of some direct sum of modules of the form $\uR{G/K}$.

For an $\RG$-module $M$, there exists a surjective map
$$P=\bigoplus_{H\in\cF}(\bigoplus_{m \in B_H} P_H)\twoheadrightarrow M$$
where $B_H$ is a basis of $M(H)$ as an $R$-module. The kernel of this surjective map is again an $R\G$-module and we can find a surjective map of a projective module onto the kernel. Thus, every $R\G$-module $M$ admits a projective resolution $$ \cdots \to P_n\to P_{n-1} \to \cdots \to P_2 \to P_1 \to P_0 \to M.$$ By standard methods in homological algebra we can show that any two projective resolutions of $M$ are chain homotopy equivalent.

The $\RG$-module category has enough injective modules as well and for given $\RG$-modules $M$ and $N$, the ext-group $\Ext _{\RG} ^n (M, N)$ is defined as the $n$-th cohomology of the cochain complex $\Hom _{\RG} (M, I^* )$ where $N \to I^*$ is an injective resolution of $N$. Since we also have enough
projectives, the ext-group $\Ext _{\RG} ^n (M, N)$ can also be
calculated using a projective resolution of $M$. We have the following:

\begin{proposition}\label{pro:Ext groups via proj res} Let $M$ and $N$ are $\RG$-modules. Then, for each $n\geq 0$, we have $$\Ext _{\RG}^n (M, N)
\cong H^n (\Hom _{\RG}(P_*, N))$$ where $P_*$ is a projective resolution of $M$ as an $R\G$-module.
\end{proposition}

\begin{proof} This follows from the balancing theorem in homological algebra. Take an injective resolution $I^*$ for $N$ and consider the double complex $\Hom _{\RG} (P_*, I^*)$. Filtering this double complex in two different ways and by calculating the corresponding spectral sequences, we get the desired isomorphism.
\end{proof}

When $\cF=\{1\}$, the ext-group $\Ext_{\RG}^n (M, N)$ is the same as the usual ext-group $$\Ext ^n _{RG} (M (1), N(1))$$ over the group ring $RG$. So, the ext-groups over group rings, and hence the group cohomology, can be expressed as the ext-group over the orbit category for some suitable choices of $M$ and $N$. In the rest of the section we prove Theorem \ref{thm:main 1; computational tool} which says that this is also true for the relative cohomology of a group. 

Let $\un R$ denote the $\RG$-module which takes the value $\un R(H)=R$ for every $H \in \cF$ and such that for every $f: G/K\to G/H$, the induced map $\underline R(f): \un R (H)\to \un R (K)$ is the identity map. Given $\RG$-modules $M$ and $N$, the tensor product of $M$ and $N$ over $R$ is defined as the $\RG$-module such that for all $H \in\cF$, $$ (M\otimes _R N )(H)= M(H)\otimes _R N(H)$$ and the induced map is $(M \otimes _R N)(f)=M(f)\ot _R N(f)$ for every $f: G/K \to G/H$. Note that the module $\un R$ is the identity element
with respect to tensoring over $R$, i.e, $M \ot _R \un R=\un R \ot_R M =M$ for every $\RG$-module $M$. We also have the following:

\begin{lemma}\label{lem:proj tensor proj}
If $P$ and $Q$ are projective $\RG$-modules, then $P \ot _R Q$ is also projective.
\end{lemma}

\begin{proof} Since every projective module is a direct summand of a free module $\oplus_i \uR{G/H_i}$, it is enough to prove this statement for module of type $\uR{G/H}$. Since $$\uR{G/H}\ot _R \uR{G/K}=\bigoplus _{HgK\in H\backslash G/K} \uR{G/ (H \cap \leftexp{g} K)}$$ and since $\cF$ is closed under conjugations and taking subgroups, this tensor product is also projective.
\end{proof}

This is used in the proof of the following proposition.

\begin{proposition}[Theorem 3.2, \cite{nucinkis2}]\label{pro:proj resolution at G/1}
Let $P_*$ be a projective resolution of $\un R$ as an
$\RG$-module. Then, $P_*(1)$ is a relative $\cF$-projective resolution of the trivial $RG$-module $R$.  
\end{proposition}

\begin{proof}
If we apply $-\otimes_R \uR{G/H}$ to the resolution $P_*\to \un R$, then we get $$\cdots \to P_n \ot _R \uR{G/H} \maprt{\bd_n\ot \id} P_{n-1}\otimes_R \uR{G/H} \to \cdots \to P_0 \otimes_R \uR{G/H} \maprt{\bd _0 \otimes \id}  \uR{G/H} \to 0.$$ By Lemma \ref{lem:proj tensor proj}, all the modules in this sequence are projective. So, the sequence splits. This means that for every $n\geq 0$, the short exact sequence $$ 0 \to \ker (\bd _n \ot \id) \to P_n \ot _R \uR{G/H} \to \image (\bd _n \ot \id) \to 0$$ splits. If we evaluate this sequence at $1$, we get a split sequence of $RG$-modules. This implies that the sequence $$ 0 \to \ker \bd_n \to P_n (1) \to \image \bd _n \to 0$$ is $\cF$-split for all $n\geq 0$. Note also that $P_n (1)$ is a direct summand of $F(1)$ for some free $\RG$-module $F$. So, by
Proposition \ref{pro:proj is direct summand of free}, the
$RG$-module $P_n (1)$ is $\cF$-projective. Hence, the resolution
$$\cdots \to P_n (1) \maprt{\bd_n} P_{n-1}(1) \to \cdots \to P_1 (1) \maprt{\bd _1} P_0 (1) \maprt{\bd _0} R \to 0 $$ is a relative $\cF$-projective resolution of
$R$.
\end{proof}

Now, recall that for every $RG$-module $M$, there is an $R\G$-module denoted by $M^{?}$ which takes the value $M^H$ for every $H \in \cF$ where $M^H$ denotes the $R$-submodule $$M^H=\{ m \in M \ | \ hm=m \text{ for all } h \in H \}$$ of $M$.
Note that $M^H\cong \Hom _R (R[G/H], M )$. In fact, we can choose a canonical isomorphism and we can think of the element  $m\in M^H$ as an $R$-module homomorphisms $R[G/H]\to M$ which takes $H$ to $m$. For each
$G$-map $f: G/K \to G/H$, the induced map $M(f) : M^H \to M^K$ is defined as the composition of corresponding homomorphisms with the linearization of $f$ which is
$Rf: R[G/K]\to R[G/H]$. The module $M^?$ has
the following important property:

\begin{lemma}\label{lem:important property of M^?}
Let $M$ be an $RG$-module and $M^?$ be the $\RG$-module defined above. For any projective $\RG$-module $P$, we have
$$\Hom_{\RG}(P ,M^?)\cong\Hom_{RG}(P(1),M).$$
\end{lemma}

\begin{proof} It is enough to prove the statement for $P=\uR{G/H}$ for some $H \in \cF$. Note that we have $$\Hom _{\RG} (\uR{G/H}, M^?)\cong M^H \cong \Hom _{RG} (R[G/H], M),$$ so the statement holds in this case.
\end{proof}

Now, we are ready to prove the main theorem of this section.

\begin{proof}[Proof of Theorem \ref{thm:main 1; computational tool}]
Let $P_* \to \un R$ be a projective resolution of $\un R$ as an $\RG$-module. The ext-group $\Ext ^n _{\RG} (\un R, M^?)$ is defined as the $n$-th cohomology of the cochain complex $\Hom _{\RG} (P_*, M^?)$. By Lemma \ref{lem:important property of M^?}, we have $$\Hom
_{\RG } (P_* , M^? )\cong \Hom _{RG } (P_* (1), M)$$ as cochain
complexes. By Proposition \ref{pro:proj resolution at G/1}, the
chain complex $P_* (1)$ is a relative $\cF$-projective resolution. So,
by the definition of relative group cohomology, we get $$ \Ext
^n _{\RG} (\un R, M^? ) \cong \cF H^n (G, M)$$ as desired.
\end{proof}

\section{Periodicity of relative cohomology}
\label{sect:computation}

In this section, we consider the following question: Let $G$ be a finite group of rank $r$ and $\cF$ be the family of all subgroups $H$ of $G$ such that $\rk H \leq r-1$. Then, does there exist an $\cF$-split exact sequence of the form
$$ 0 \to \bbZ \to \bbZ X_n \to \cdots \to \bbZ X_2 \to \bbZ X_1 \to \bbZ X_0 \to \bbZ \to 0$$ where each $X_i$ is a $G$-set with isotropy in $\cF$\,? The existence of such a sequence came up as question in the process of constructing group actions on finite complexes homotopy equivalent to a sphere with a given set of isotropy subgroups. Note that the $\cF$-split condition, in fact, is not necessary for realizing a permutation complex as above by a group action, but having this condition guarantees the existence of a weaker condition that is necessary for the realization of such periodic resolutions by group actions. Note also that for constructions of group actions, algebraic models over the orbit category are more useful than chain complexes of permutation modules. For more details on the construction of group actions on homotopy spheres, see \cite{hpy} and \cite{pamuk}.

The main aim of this section is to show that the answer to the above question is negative. For this, we consider the group $G=\bbZ /2 \times \bbZ /2=\la a_1, a_2\ra$ and take $\cF=\{1, H_1, H_2, H_3\}$ where $H_1=\la a_1 \ra$, $H_2=\la a_1a_2 \ra$, and $H_3=\la a_2\ra.$ Note that if there is an exact sequence of the above form, then by splicing the sequence with itself infinitely many times, we obtain a periodic relative $\cF$-projective resolution of $\bbZ$ as
a $\bbZ G$-module. But, then the relative
cohomology $\cF H^* (G, \bbF _2)$ would be periodic. We explicitly
calculate this relative cohomology and show that it is not periodic, hence prove Theorem \ref{thm:main 2; not periodic}.

From now on, let $G$ and $\cF$ be as above and let
$R=\bbF_2$. By Theorem \ref{thm:main 1; computational tool}, we have $$\cF H^*(G,R)\cong \Ext^*_{\RG}(\underline{R},R^?).$$ Note that $R^?=\underline{R}$, so we need to calculate the ext-groups $\Ext^n_{\RG} (\un R, \un R )$ for each $n\geq 0$. To calculate these ext-groups, we consider some long exact sequences of ext-groups coming from short exact sequences of $\RG$-modules.

Let $R_0$ denote the $\RG$-module where $R_0 (1)=R$ and $R_0
(H_i)=0$ for $i=1,2,3$. Also consider, for each $i=1,2,3$, the
module $R_{H_i}$ which is defined as follows: We have $R_{H_i}
(1)=R_{H_i} (H_i)=R$ with the identity map between them and 
$R_{H_i} (H_j)=0 $ if $i\neq j$. For each $i=1,2,3$, there is an $\RG$-homomorphism $\varphi_i: R_0 \to R_{H_i}$ which is the identity map at $1$ and the zero map at other subgroups. We can give a
picture of these modules using the following diagrams:
\Small
$$\un{R}=\vcenter{\xymatrix{
R\ar@{-}[dr]&R\ar@{-}[d]&R\ar@{-}[dl]\\
&R& }} \qquad \qquad R_0 =\vcenter{\xymatrix{
0\ar@{-}[dr]&0\ar@{-}[d]&0\ar@{-}[dl]\\
&R& }}  $$
$$\qquad R_{H_1}=\vcenter{\xymatrix{
R\ar@{-}[dr]&0\ar@{-}[d]&0\ar@{-}[dl]\\
&R& }}\qquad R_{H_2}=\vcenter{\xymatrix{
0\ar@{-}[dr]&R\ar@{-}[d]&0\ar@{-}[dl]\\
&R& }}\qquad R_{H_3}=\vcenter{\xymatrix{
0\ar@{-}[dr]&0\ar@{-}[d]&R\ar@{-}[dl]\\
&R& }}$$
\normalsize
where each line denotes the identity map $\id: R\to R$ if it is from $R$ to $R$ and denotes the zero map otherwise.

Now consider the short exact sequence
\begin{equation}\label{eqn:short exact sequence}
0\to R_0 \oplus R_0 \maprt{\gamma}  R_{H_1} \oplus R_{H_2} \oplus
R_{H_3} \maprt{\pi} \un{R}\to 0
\end{equation} where $\pi$ is the
identity map at each $H_i$ and at $1$, it is defined by
$\pi (1)(r,s,t)=r+s+t$ for every $r,s,t\in R$. The map $\gamma$ is the zero map at every $H_i$ and at 1 it is the map defined by 
$$\gamma(1) (u,v)=(-u,\ u+v,\ -v).$$ Note that with respect to the direct sum decomposition above, we can express $\gamma$ with the matrix
$$\gamma=\left[\begin{matrix}-\gamma_1 & 0 \\ \gamma _2 & \gamma _2\\ 0 & -\gamma _3 \end{matrix}\right]$$
where $\gamma_i: R_0 \to R_{H_i}$ are the maps defined above. We will be using the short exact sequence given in \eqref{eqn:short exact sequence} in our computations. We start our computations with an easy
observation:

\begin{lemma}\label{lem:comp for R_0} For every $n\geq 0$, we have $\Ext^n_{\RG}(R_0,\un{R})\cong H^n(G,R).$
\end{lemma}

\begin{proof} By definition $\Ext^n_{\RG}(R_0,\un{R})=
H^n(\Hom_{\RG}(P_*,\un{R}))$ where $P_*\to R_0$ is a projective resolution of $R_0$ as an $\RG$-module. Since the definition is independent from the projective resolution that is used, we can pick a specific resolution. Let $F_*$ be a free resolution of $R$ as an $RG$-module. Take $P_*$ as the resolution where $P_*(1)=F_*$ and $P_*(H_i )=0$ for $i=1,2,3$. If $F_k = \oplus _{n_k} RG$, then
$P_k=\oplus _{n_k} \uR{G/1}$, so $P_*$ is a projective resolution of
$R_0$. Since $\Hom _{\RG} (P_* , \un R )\cong \Hom _{RG} (F_*, R)$,
the result follows.
\end{proof}

\begin{lemma}\label{lem:comp for R_H} If $H=H_i$ for some $i\in \{ 1,2,3\}$, then $\Ext^n_{\RG}(R_{H} ,\un{R})
\cong H^n(G/{H} ,R)$ for every $n \geq 0$.
\end{lemma}

\begin{proof} Take a free resolution of $R$ as an $R[G/H]$-module
$$F_* :  \ \cdots \to \oplus _{m_2} R[G/H] \to \oplus _{m_1} R[G/H] \to \oplus _{m_0} R[G/H] \to R\to 0.$$ We can consider the same resolution a resolution of $R$ as an $RG$-module via the quotient map $G \to G/H$. The resolution we obtain is the inflation of $F_*$ denoted by $\inf ^G_{G/H} F_*$. Define a projective resolution $P_*$ of $R_H$ as an $R\G$-module by
taking $P_* (H)=F_*$, $P_* (1)=\inf ^G _{G/H} F_*$, and $P_*
(K)=0$ for other subgroups $K \in \cF$. There is only one nonzero restriction map $P_*(H)\to P_* (1)$. Assume that this map is given by the inflation map. For each $n\geq 0$, the $\RG$-module $P_n$ is isomorphic to $\oplus _{m_n} \uR{G/H}$, so $P_*$ is a projective resolution of $R_H$ as an $\RG$-module. Note that $$\Hom _{\RG }
(\uR{G/H}, \un R)\cong \Hom _{R[G/H] } ( R[G/H], R).$$ So, applying
$\Hom _{\RG} (-, \un R)$ to $P_*$, we get
$$\Hom _{\RG } (P_* , \un R)\cong \Hom _{R[G/H]} (F_*, R)$$ as
cochain complexes. So, the result follows.
\end{proof}

\begin{lemma}\label{lem:the map j^*}
For every $i\in \{1,2,3\}$, let $\gamma_i ^* : \Ext ^n _{\RG }
(R_{H_i} , \un R) \to \Ext ^n_{\RG} (R_0, \un R)$ denote the map induced by $\gamma_i: R_0 \to R_{H_i}$ defined above. Then, $\gamma_i ^*$ is the same as the inflation map $\inf ^G _{G/H_i}: H^n (G/H_i, R) \to H^n (G, R)$ in group cohomology under the isomorphisms given in the previous two lemmas.
\end{lemma}

\begin{proof} Let $P_*$ and $Q_*$ be projective resolutions of $R_0$ and $R_{H_i}$, respectively. We can assume that they are in the form  as in the proofs of the above lemmas.
In particular,
we can assume $P_* (1)$ is a free resolution of $R$ as an
$RG$-module and $Q_* (1)$ is the inflation of a free resolution of $R$ as an
$R[G/H_i]$-module. The identity map on $R$ lifts to a chain map $f'_*: P_* (1) \to Q_* (1)$ since $P_*(1)$ is a projective resolution and $Q_*(1)$ is acyclic. This chain map can be completed (by taking the zero map at other subgroups) to a chain map $f_* : P_* \to Q_*$ of $R\G$-modules. The map $\gamma_i^*$ between the ext-groups is the map induced by this chain map. 
But the map induced by $f'_*$ on cohomology is the inflation map $\inf ^G _{G/H_i}: H^n (G/H_i, R) \to H^n (G, R)$ by the definition of the inflation map in group cohomology.  So, the result follows.
\end{proof}

Now, we are ready to prove the main result of this section.

\begin{proof}[Proof of Theorem \ref{thm:main 2; not periodic}] Consider the following long exact sequence of ext-groups coming from the short exact sequence given in
\eqref{eqn:short exact sequence}:
\begin{eqnarray*}
\cdots \to \Ext_{\RG}^{n-1}(\oplus_i R_{H_i},\un{R})
\maprt{\gamma^*} \Ext_{\RG}^{n-1}(\oplus_2R_0,\un{R})
\maprt{\delta}\Ext_{\RG}^n(\un{R},\un{R})  \\  \maprt{\pi^*}
\Ext_{\RG}^n(\oplus_iR_{H_i},\un{R})\maprt{\gamma^*}
\Ext_{\RG}^n(\oplus_2R_0,\un{R})\to \cdots
\end{eqnarray*}
By Lemma \ref{lem:comp for R_0} and \ref{lem:comp for R_H}, we have $$\Ext ^n _{\RG} (\oplus _i
R_{H_i}, \un R)\cong \oplus _i H^n (G/H_i, R) \quad \text{and} \quad \Ext^n _{\RG} (\oplus _2 R_0, \un R)\cong \oplus _2 H^n(G, R)$$ for all $n\geq 0$. It is well-known that $H^*(C_2, R)\cong R[t]$ for some one-dimensional class $t\in H^1(G, R)$. Let $t_1,t_2,t_3$ be the generators of cohomology rings $H^* (G/H_i, R)$ for $i=1,2,3$, respectively. By Kunneth's theorem $H^*(G, R)\cong R[x,y]$ for some $x,y \in H^1 (G, R)$. Let us choose $x$ and $y$ so that $x=\inf ^G _{G/H_1} (t_1)$ and $y=\inf ^G _{G/H_3} (t_3)$. Then, we have $\inf^G_{G/H_2}(t_2)=x+y$. Note that
$$\gamma^*=\left[\begin{matrix}-\gamma_1 ^* & \gamma _2 ^* & 0 \\ 0 & \gamma _2 ^* & -\gamma ^* _3 \end{matrix}\right]$$
and by Lemma \ref{lem:the map j^*}, we have $\gamma _i ^*=\inf _{G/H_i} ^G$ for all $i=1,2,3$. Therefore, we obtain
$$\gamma ^*(t_1)=(-x, 0), \ \gamma ^*(t_2)=(x+y,x+y),\ \text{and}\ \gamma^*(t_3)=(0, -y).$$ From this it is easy to see that $$\gamma^*: \Ext_{\RG}^n(\oplus_iR_{H_i},\un{R})\to
\Ext_{\RG}^n(\oplus_2R_0,\un{R})$$ is injective for $n\geq 1$, so we get short exact sequences of the form
$$0\to \oplus _i H^{n-1}(G/H_i,R)\maprt{\gamma^*} \oplus _2 H^{n-1} (G, R)\maprt{\delta}\Ext_{\RG}^{n}(\un{R},\un{R})\to 0$$
for every $n \geq 2$. This gives that $$d_n=\dim _R \Ext _{\RG} ^n (\un R, \un R)=2n-3$$ for $n\geq 2$. Looking at the dimensions $n=0,1$ more closely we obtain that $$d_n= (1,0,1,3,5,7,9,\dots).$$  So, $\cF H^n (G, R)=\Ext_{\RG} ^n(\un{R},\un{R})$ is not periodic.
\end{proof}

\section{Limit functor between two families of subgroups}
\label{sect:limitfunctor}

Let $G$ be a finite group and $\cF$ be a family of subgroups closed under conjugation and taking subgroups. Let $\G$ denote the orbit category $\OrG$. An $\RG$-module $M$ is a contravariant functor from $\G$ to category of $R$-modules, so we can talk about the inverse limit of $M$ in the usual sense.
Recall that the inverse limit of $M$ denoted by $$\lim \limits _{\underset{H\in \cF} \longleftarrow} M$$ is defined as the $R$-module of tuples $(m_H) _{H \in \cF} \in \prod _{H \in \cF} M(H)$ satisfying the condition $M(f) m_H=m_K$ for every $G$-map $f: G/K \to G/H.$ To simplify the notation, from now on we will denote the inverse limit of $M$ with $\lim _{\cF} M.$ Our first observation is the following:

\begin{lemma}\label{lem:lim and Hom} Let $M$ be an $\RG$-module. Then, $\lim _{\cF} M\cong \Hom _{\RG} (\underline R, M )$.
\end{lemma}

\begin{proof} This follows from the definition of $\Hom$ functor in $\RG$-module category (see \cite[Proposition 5.1]{webb} for more details).
\end{proof}

Now we define a version of inverse limit for two families. Relative limit functors are also considered in \cite{symonds} and some of the results that we prove below are already proved in the appendix of \cite{symonds} but we give more details here. 

\begin{definition}\label{def:relativelimit}
Let $\cV \subseteq \cW$ be two families of subgroups of $G$ which are closed
under conjugation and taking subgroups.  Let $\G_{\cV}=\Or_{\cV}(G)$
and $\G_{\cW}=\Or_{\cW}(G)$. Then define
$${\rm lim}_{\cV}^{\cW}: R\G_{\cV}\to R\G_{\cW} $$  as the functor which takes the value $$({\rm lim}_{\cV}^{\cW}M)(H) =\Hom_{R\G_{\cV}}(R[G/H^?],M)$$ at every $H \in \cW$ and the induced maps by $G$-maps $f: G/K \to G/H$ are given by usual composition of homomorphisms with the linearization of $f$. 
\end{definition}

The description of $(\lim _{\cV } ^{\cW} M ) (H)$ given above comes from the desire to make it right adjoint to the restriction functor $$\Res ^{\cW} _{\cV}: R \G _{\cW} \to R \G _{\cV}$$ which is defined by restricting the values of an $R\G_{\cW}$-module to the smaller family $\cV$. Note that the adjointness gives
$$({\rm \lim}_{\cV}^{\cW}M)(H) \cong \Hom _{R\G_{\cW} } (R[G/H ^? ], {\rm lim}_{\cV} ^{\cW} M)\cong \Hom_{R\G_{\cV}}(\Res ^{\cW}_{\cV} R[G/H^?],M),$$ and that is why $\lim _{\cV} ^{\cW} M$ is defined as above. We also have the following natural description in terms of the usual meaning of inverse limits.

\begin{proposition}\label{pro:limitdescription} Let $\cW$ and $\cV$ be as above and $H \in \cW$. Then,  $(\lim ^{\cW} _{\cV} M )(H)$ is isomorphic to the $R$-module of all tuples $$(m_K)_{K \in \cV|_H}\in \prod _{K \in \cV |_H } M(K)  \ \text{ where } \ \cV|_H=\{ K \in \cV \, |\, K\leq H\}$$ satisfying the compatibility conditions coming from  inclusions and conjugations in $H$. 
\end{proposition}

\begin{proof} Let us denote the $R$-module of tuples $(m_K)_{K \in \cV |_H}$ by $\lim _{\cV |_H} M$. We will prove the proposition by constructing an explicit isomorphism
$$\varphi : \Hom _{R\G _{\cV}} (R[G/H ^? ], M) \to {\rm lim} _{\cV |_H } M.$$ The $R\G_{\cV}$-module $R[G/H ^?]$ takes the value $R[(G/H)^K ]=R\{gH \ | \ K^g \leq H\}$ 
at every subgroup $K \in \cV$ with $K^g \leq H$ and takes the value zero at all other subgroups. Given a homomorphism $f=(f_K)_{K \in \cV}$ in $\Hom _{R\G _{\cV}}(R[G/H ^?], M)$, we define $\varphi (f)$ as the tuple $( f_K (H))_{K\in \cV|_H}$ where $H$ denotes the trivial coset. Note that we have $H\in (G/H)^K$ since $K\leq H$. If $L\leq K \leq H$, then it is clear that $\Res ^K _L m_K=m_L$ since $$\Res ^K _L : R[(G/H)^K ]\to R[(G/H)^L]$$ is defined by inclusion so it takes $H$ to $H$. Similarly, for every $h \in H$, we have $c^h (m_K)=m_{\leftexp{h}K}$ since $c^x: R[(G/H)^K ]\to R[(G/H)^{\leftexp{x}K}]$, which is defined by $H\to xH$, is the identity map when $x\in H$. Here $\leftexp{x}K$ denotes the conjugate subgroup $xKx^{-1}$. Therefore, the tuple $(f_K (H))_{K\in \cV|_H}$ satisfies the compatibility conditions, so $\varphi(f)$ is in $\lim _{\cV|_H}M$. 

To show that $\varphi$ is an isomorphism, we will prove that for every tuple $(m_K)_{K\in \cV |_H}$ in $\lim _{\cV |_ H } M$ there is a unique family of homomorphisms $f_L : R[(G/H )^L] \to M(L)$ which satisfy $f_L (H)=m_L$ for all $L \leq H$, and which are also compatible in the usual sense of the compatibility of homomorphisms in $\Hom _{R\G _{\cV}} (R[G/H ^?], M)$. Let $L\in \cV$ be such that $(G/H)^L\neq \emptyset$. We define the $R$-homomorphism $f_L : R[(G/H)^L] \to M(L)$ in the following way: Let $gH$ be a coset in $(G/H)^L$. Then we have $L^g \leq H$. Let $K=L^g$. Since $K \leq H$, we have a given element $m_K\in M(K)$. Set $f_L (gH)=c^g (m_K)$. Since we can do this for all $gH\in (G/H)^L$, this defines $f_L$ completely for all $L$ with $(G/H)^L \neq \emptyset$. We take $f_L=0$ for other subgroups. 

Now note that under these definitions, we have a commuting diagram 
$$\xymatrix{R[(G/H)^K]\ar[d]^-{c^g}
\ar[r]^-{f_K}& M(K) \ar[d]^-{c^g}\\
R[(G/H)^L] \ar[r]^-{f_L} &M(L) }$$
since the map on the left takes $H$ to $gH$. It is also clear that the maps $f_L$ are compatible under restrictions since
the restriction maps on $R[G/H^?]$ are given by inclusions. So, the family $f=(f_L)_{L \in \cV}$ defines a homomorphism of $R\G_{\cV}$-modules. Since the values of $f$ at each $K$ are defined in a unique way using the tuple $(m_K)_{K \in \cV |_H}$, this shows that the homomorphism $\varphi$ is an isomorphism.
\end{proof}

We now prove the adjointness property mentioned above.

\begin{proposition}[Proposition 12.2, \cite{symonds}]\label{pro:adjointness}
Let $M$ be an $R\G _{\cW}$-module and $N$ be an $R\G _{\cV}$-module. Then, we have
$$\Hom _{R\G_{\cW}} (M,\, {\rm lim}_{\cV} ^{\cW} N)\cong \Hom _{R\G _{\cV} } (\Res ^{\cW} _{\cV} M,\, N).$$
\end{proposition}

\begin{proof} Note that for $K \in \cV$, we have 
$$({\rm lim}_{\cV} ^{\cW} N )(K) = \Hom_{R\G_{\cV}}(R[G/K^?], N)\cong N(K),$$
so we can easily define an $R$-homomorphism 
$$\varphi: \Hom _{R\G_{\cW}} (M,\, {\rm lim}_{\cV} ^{\cW} N)\to \Hom _{R\G _{\cV} } (\Res ^{\cW} _{\cV} M,\, N)$$
as the homomorphism which takes an $R\G _{\cW}$-module homomorphism $\alpha: M\to {\rm lim}_{\cV} ^{\cW} N$ to an $R\G _{\cV}$-homomorphism by restricting its values to the subgroups in $\cV$. For the homomorphism in the other direction, note that for every $H \in \cW$,
$$({\rm lim}_{\cV} ^{\cW} N) (H)\cong {\rm lim} _{\cV|_H} N $$
by Propositions \ref{pro:limitdescription}, so an element of $(\lim _{\cV} ^{\cW} N) (H)$ can be thought of as a tuple $(n_K)_{K \in \cV |_H}$ with $n_K \in N(K)$. So, given a homomorphism $f\in \Hom _{R\G_{\cV} } (\Res ^{\cW} _{\cV} M, N)$, we can define a unique homomorphism $f'$ in $\Hom _{R\G _{\cW} } (M, \lim _{\cV } ^{\cW} N)$ by defining $$f'_H (m) = (f_K (\Res^H _K m ))_{K \in \cV |_H }$$ for every $m \in M(H)$ and for every $H\in \cW$. It is clear that $f'$ is uniquely defined by $f$ and that $\varphi (f')=f$. So, $\varphi$ is an isomorphism.
\end{proof}

As a consequence of this adjointness we can conclude the following:

\begin{corollary}\label{cor:injtoinj} The limit functor $\lim _{\cV}^{\cW}$ takes injective modules to injective modules.
\end{corollary}

\begin{proof} This follows from the adjointness property
given in Proposition \ref{pro:adjointness} and the fact that $\Res ^{\cW} _{\cV}$  takes exact sequences of $R\G _{\cW}$-modules to exact sequences of $R\G _{\cV}$-modules.
\end{proof}

Now we discuss some special cases of the limit functor $\lim _{\cV} ^{\cW}$. 

\begin{example}\label{ex:invariantfunctor} Let $\cV=\{ 1\}$ be the family formed by a single subgroup which is the trivial subgroup and $\cW=\cF$ be an arbitrary family of subgroups of $G$ closed under conjugation and taking subgroups. Modules over $R\G _{\{1\}}$ are the same as $RG$-modules. Let  $M$ be an
$RG$-module.  Then,
\begin{eqnarray*}
(\lim{}^{\cF}_{\{ 1\}}M)(H) \cong  \Hom_{R\G_{\{ 1\}}}(R[G/H^?],M)\cong \Hom_{RG}(R[G/H],M) \cong M^H.
\end{eqnarray*}
It is easy to check that these isomorphisms commute with restrictions and conjugations, so $\lim{}^{\cF}_{\{ 1\}}M \cong M^?$ as $R\G$-modules where $\G=\OrG$. Hence
$${\rm lim}^{\cF}_{\{ 1\}}: RG\text{-Mod} \to R\G\text{-Mod}$$ 
is the same as the invariant functor mapping $M\mapsto M^?$.
\end{example}

Another special case is the following:

\begin{example}\label{ex:toall} Let $\cV=\cF$ be an arbitrary family of subgroups of $G$ closed under conjugation and taking subgroups, and let $\cW=\{all\}$ be the family of all subgroups of $G$. Then for every $R\G_{\cV}$-module $M$ we have
\begin{eqnarray*}
({\rm lim}_{\cF}^{\{all\}}M)(G)\cong \Hom_{R\G}( R[G/G^?],M)\cong\Hom_{R\G}(\un{R},M)\cong{\rm lim}_{\cF}M.
\end{eqnarray*}
So, we can write the usual limit functor as the composition 
$${\rm lim}_{\cF}M=ev_G\circ {\rm lim}_{\cF}^{\{all\}}$$ where
$ev_G:R\G_{\{all\}}\to R$-Mod is the functor defined by $ev_G(M)=M(G)$. 
\end{example}

We have the following easy observation for the composition of limit functors.

\begin{lemma}\label{lem:composition}
Let $G$ be a finite group and $\cU \subseteq \cV \subseteq \cW$ be three families of subgroups of $G$ which are closed under conjugation and taking subgroups. Then we have $${\rm lim} _{\cU} ^{\cW} ={\rm lim}_{\cV} ^{\cW} \circ {\rm lim} _{\cU }^{\cV}.$$ In particular, for any family $\cF$ the composition $\lim_{\cF}\circ\lim^{\cF}_{\{ 1\}}$ is the same as the functor $RG\text{-Mod} \to R\text{-Mod}$ which takes $M\to M^G$.
\end{lemma}

\begin{proof} The first statement follows from the fact that $\Res ^{\cW}_{\cU}=\Res ^{\cV} _{\cU} \Res ^{\cW} _{\cV}$ and the adjointness of limit and restriction functors. The second statement is clear since $\lim _{\{ 1\}} M \cong M^G$ for every $RG$-module $M$. 
\end{proof}

Recall that the cohomology of group $H^n(G,R)$ is defined as the $n$-th derived functor of the $G$-invariant functor $M\to M^G$. So, it makes sense to look at the derived functors of the limit functor as a generalization of group cohomology.

\section{Higher Limits and relative group cohomology}
\label{sect:higherlimits}
Let $G$ be a finite group and $\cF$ be a family of subgroups of $G$ closed under conjugation and taking subgroups. Let $R$ be a commutative ring and $\G$ denote the orbit category $\OrG$. For an $R\G$-module $M$, the usual inverse limit $\lim _{\cF}M$ is isomorphic to $\Hom _{R\G} (\underline R, M)$. Since the $\Hom$ functor is a left exact functor, the limit functor $M \to \lim _{\cF} M$ is also left exact. So we can define its right derived functors in the usual way by taking an injective resolution of $M \to I^*$ in the $\RG$-module category and then defining the $n$-th derived functor of the inverse limit functor as $${\rm lim}_{\cF} ^n M:= H^n ({\rm lim}_{\cF} I^* ).$$ This cohomology group is called the $n$-th {\it higher limit} of $M$. As a consequence of the isomorphism in Lemma \ref{lem:lim and Hom}, we have  $\lim ^n _{\cF} M\cong \Ext ^n _{\RG} (\underline R, M)$ so higher limits can be calculated also by using a projective resolution of $\underline R$ (see Proposition \ref{pro:Ext groups via proj res}). Higher limits have been studied extensively since they play an important role in the calculation of homotopy groups of homotopy colimits. For more details on this we refer the reader to \cite{grodal} and \cite{symonds}.

The situation with $\lim _{\cF}$ can be extended easily to the limit functor with two families. Let $\cV \subseteq \cW$ be two families of subgroups of $G$ which are closed under conjugation and taking subgroups. Note that for each $H \in \cW$, we have  $(\lim _{\cV} ^{\cW} M )(H) =\Hom _{R\G _{\cV} } (R[G/H] ^?, M),$ so $\lim _{\cV} ^{\cW}$ is left exact at each $H$, hence it is left exact as a functor $R\G _{\cV}\text{-Mod}\to R\G _{\cW}\text{-Mod}.$ This leads to the following definition.

\begin{definition} For each $n \geq 0$, the $n$-th higher limit $(\lim_{\cV}^{\cW})^n$ is defined as the $n$-th
derived functor of the limit functor ${\rm lim}_{\cV}^{\cW}$. So, for every $R \G_{\cV}$-module $M$ and for every $n\geq 0$, we have
$$({\rm lim}_{\cV}^{\cW})^n(M):=H^n({\rm \lim}_{\cV}^{\cW}I^*)$$ where $I^*$ is an injective resolution of $M$ as an $R\G_{\cV}$-module.
\end{definition}

The special cases of the limit functor that were considered above in Examples \ref{ex:invariantfunctor} and \ref{ex:toall} have higher limits which correspond to some known cohomology groups.

\begin{proposition}\label{pro:groupcohomology}
Let $G$ be a finite group, $\cF$ be a family of subgroups of $G$ closed under conjugation and taking subgroups, and let $M$ be an $RG$-module. Then, for every $n \geq 0$, the functor $(\lim _{\{1\}}^{\cF} )^n (M)$ is isomorphic to the group cohomology functor $$H^n (?, M) : RG\text{-Mod} \to R\G\text{-Mod} $$ which has the value $H^n (H, M)$ for every subgroup $H \in \cF$.
\end{proposition}

\begin{proof} By the definition given above, we have $$({\rm lim} _{\{1\}}^{\cF} )^n (M)=H^n ({\rm lim} _{\{1\}} ^{\cF} I^* )=H^n ( (I^*)^?).$$ So, for each $H \in \cF$, the $n$-th higher limit has the value $H^n ((I^* )^H)=H^n (H, M)$. The fact that these two functors are isomorphic as $R\G$-modules follows from the definition of restriction and conjugation maps in group cohomology.
\end{proof}

We also have the following:

\begin{proposition}\label{pro:relativecohomology} Let $G$ and $\cF$ be as above and let $M$ be an $RG$-module. Then, for every $n\geq 0$, the higher limit $\lim _{\cF} ^n (M^?)$ is isomorphic to the relative cohomology group $\cF H^n (G, M)$.
\end{proposition}

\begin{proof}
We already observed that $\lim _{\cF } ^n (M^?)\cong \Ext^n _{\RG } (\underline R, M^?) $. So, the result follows from Theorem \ref{thm:main 1; computational tool}.
\end{proof}

Now, we will construct a spectral sequence that converges to the cohomology of a given group $G$ and which has the horizontal line isomorphic to the relative group cohomology of $G$. For this we first recall the following general construction of a spectral sequence, called the Grothendieck spectral sequence.

\begin{theorem}[Theorem 12.10, \cite{mccleary}] Let $\cC_1$, $\cC_2$, $\cC_3$ be abelian categories and $F:\cC_1\to\cC_2$ and $G:\cC_2\to\cC_3$ be covariant functors.  Suppose $G$ is left exact and $F$ takes injective objects in $\cC_1$ to $G$-acyclic objects in $\cC_2$.  Then there is a spectral sequence with $$E^{p,q}_2\cong (R^pG)(R^qF(A))$$ and converging to $R^{p+q}(G\circ F)(A)$ for $A\in\cC_1$.
\end{theorem}

Here $R^nF$ denotes the $n$-th right derived functor of a functor $F$. Also recall that an object  $B$ in $\cC_2$ is called $G$-acyclic if
$$R^nG(B)=\left\{\begin{array}{ll}
G(B), &\, n=0\\0, & \,  n\geq 1.
\end{array} \right.
$$

Now we will apply this theorem to the following situation:  Let $\cU \subseteq \cV \subseteq \cW$ be three families of subgroups of $G$ which are closed under conjugation and taking subgroups. Consider the composition
 $$\lim{}_{\cV}^{\cW}\circ\lim{}_{\cU}^{\cV}:R\G_{\cU}\to R\G_{\cW}.$$ By Lemma \ref{lem:composition}, this composition is equal to $\lim _{\cU} ^{\cW}$. We also know from the discussion at the beginning of the section that the limit functor is left exact and by Corollary \ref{cor:injtoinj} we know that it takes injectives to injectives.
So, we can apply the theorem above and conclude the following:

\begin{theorem}\label{thm:mostgeneralform} Let $G$ be a finite group and $\cU \subseteq \cV \subseteq \cW$ be three families of subgroups of $G$ which are closed under conjugation and taking subgroups. Then, there is a first quadrant spectral sequence 
$$E_2 ^{p,q}= ({\rm lim}_{\cV}^{\cW} )^p({\rm lim}^{\cV}_{\cU})^q(M)\Rightarrow ({\rm lim} _{\cU} ^{\cW})^{p+q}(M).$$
\end{theorem}

The spectral sequence given in Theorem \ref{thm:main 3; spect. seq.} is a special case of the spectral sequence given above. To obtain the spectral sequence in Theorem \ref{thm:main 3; spect. seq.}, we take
$\cW=\{all\}$, $\cV=\cF$, and $\cU=\{1\}$ and evaluate everything at $G$. Then, the spectral sequence in Theorem \ref{thm:mostgeneralform}  becomes $$E_2^{p,q}=({\rm lim}_{\cF})^p ({\rm lim}^{\cF} _{\{1\}})^q (M)\Rightarrow ({\rm lim} 
_{\{ 1\}} )^{p+q} (M).$$  
Using Propositions \ref{pro:groupcohomology} and  \ref{pro:relativecohomology}, we can replace all the higher limits above with more familiar cohomology groups. As a result we obtain a spectral sequence
$$ E_2^{p,q}=\Ext_{R\G} ^p (\underline R , H^q (? , M) )\Rightarrow H^{p+q}(G,M).$$
Note that for $q=0$, we have $E_2 ^{p,0}= \Ext ^p _{R\G} (\underline R, M^?)\cong\cF H^q (G, M)$ by Theorem \ref{thm:main 1; computational tool}. So, the proof of Theorem \ref{thm:main 3; spect. seq.} is complete.

\begin{remark} The spectral sequence in Theorem \ref{thm:main 3; spect. seq.}
can also be obtained as a special case of a Bousfield-Kan cohomology spectral sequence of a homotopy colimit. Note that since the subgroup families that we take always include the trivial subgroup, they are ample collections, and hence the cohomology of the homotopy colimit of classifying spaces of subgroups in the family is isomorphic to the cohomology of the group. More details on this can be found in \cite{dwyer}.
\end{remark}

Note that if we consider $E_2^{p,q}$ with $p=0$, then we get
$$E_2 ^{0,q}= \Ext ^0 _{R\G} (\underline{R}, H^q(?, M))=\Hom _{R\G} (\underline{R}, H^q (?, M))=\lim_{\underset{H\in \cF} \longleftarrow } H^q (H, M).$$  This suggests the following proposition:

\begin{proposition}\label{pro:edgevertical} The edge homomorphism $$H^*(G, M) \twoheadrightarrow E_{\infty} ^{0,*} \rightarrowtail E_2^{0,*} \cong \lim_{\underset{H\in \cF} \longleftarrow } H^* (H, M)$$
of the spectral sequence in Theorem \ref{thm:main 3; spect. seq.} is given by the map $u\to (\Res ^G _H u)_{H\in \cF}$.
\end{proposition}

\begin{proof} Note that for every $H\in \cF$, we can define $\Gamma _H=\OrH$ as the restriction of the orbit category $\Gamma_G =\OrG$ to $H$. The spectral sequence in Theorem \ref{thm:main 3; spect. seq.} for $H$ is of the form 
$$ E_2^{p,q}=\Ext_{R\G_H} ^p (\underline R , H^q (? , M) )\Rightarrow H^{p+q}(H,M).$$ Since $\underline R=R[H/H^?]$ is a projective $R\G _H$-module, we have $E_2^{p,q}=0$ for all $p>0$, so the edge homomorphism to the vertical line is an isomorphism. Now the result follows from the comparison theorem for spectral sequences. 
\end{proof}

For the other edge homomorphism first observe that there is a
natural homomorphism from relative group cohomology to the usual group cohomology $$\varphi: \cF H^n (G, M) \to H^n (G,M)$$ defined as follows: Let $P_*$ be a projective resolution of $R$ as an $RG$-module and $Q_*$ be a relative $\cF$-projective resolution of $R$. Since $P_*$ is a projective resolution and $Q_*$ is acyclic, by the fundamental theorem of homological algebra, there is a chain map $f_*: P_* \to Q_*$. This chain map induces a chain map $\Hom _{RG} (Q_*, M) \to \Hom _{RG} (P_*, M)$ of cochain complexes and hence a group homomorphism $\varphi: \cF H^n (G, M) \to H^n (G,M)$. The chain map $f_*$ is unique up to chain homotopy so the induced map $\varphi$ does not depend on the choices we make.

Alternatively, one can take an injective resolution $I^*$ of $M$ as an $RG$-module and an injective resolution $J^*$ of $M^?$ as an $R\G$-module. Since $(I^*)^?$ is still an injective resolution (but not exact anymore), we have a chain map $J^* \to (I^*)^?$ of $R\G$-modules which induces the identity map on $M^?$. Applying the functor $\Hom _{R\G} (\underline{R}, -)$ to this map, we get a map $\cF H^n (G, M) \to H^n (G, M)$. Note that this map is the same as the map $\varphi$ defined above. One can see this easily as a consequence of the balancing theorem in homological algebra which allows us to calculate ext-groups using projective or injective resolutions. Now we can prove the following: 

\begin{proposition}\label{pro:edgehorizontal} The edge homomorphism $$\cF H^*(G, M)\cong E_2 ^{*,0}  \twoheadrightarrow E_{\infty} ^{*,0} \rightarrowtail H^* (G, M)$$
of the spectral sequence in Theorem \ref{thm:main 3; spect. seq.} is given by the map $\varphi$ defined above.
\end{proposition}

\begin{proof} Let $M \to I^*$ be an injective resolution of $M$ as an $RG$-module. Applying the limit functor $\lim _{\{ 1\}} ^{\cF}$ to $I^*$, we get a cochain complex $(I^*)^?$ of $R\G$-modules. Note that by construction there is a chain map
$M^? \to (I^* )^? $ where $M^?$ is a chain complex concentrated at zero. In the construction of Grothendieck spectral sequence, one takes a injective resolution of the cochain complex $(I^*)^?$ to obtain a double complex $C^{*,*}$ where for each $q$, $$0\to (I^q)^? \to C^{0,q} \to C^{1,q} \to \cdots $$  is an injective resolution of $(I^q)^?$. Let $M^? \to J^*$ be an injective resolution of $M^?$ as an $R\G$-module. By the fundamental theorem of homological algebra, there is chain map $J^*\to C^{*,*}$ which comes from a chain map towards the bottom line of the double complex. When we apply $\lim _{\cF}$ to this chain map, we obtain a map of cochain complexes $\lim _{\cF} J^*\to \lim _{\cF} C^{*,*}$ and the edge homomorphism is the map induced by this chain map. Since the total complex of the double complex $C^{*,*}$ is chain homotopy equivalent to $(I^*)^?$,  we obtain that the edge homomorphism is induced by a chain map $\lim _{\cF} J^*
\to \lim _{\cF}(I^*)^?=(I^*)^G $ where $I^*$ is an injective resolution of $M$ as an $RG$-module and $J^*$ is an
injective resolution of $M^?$ as an $R\G$-module. Note that this chain map is defined in the same way as the chain map that induces the map $\varphi$. Since any two chain maps $J^* \to (I^*)^?$ are chain homotopy equivalent, the edge homomorphism is the same as the map $\varphi$. 
\end{proof}

\begin{corollary}\label{cor:edgehom} Let $R_0$ denote the $R\G$-module with the value $R$ at $1$ and the value zero at every other subgroup.
Then the edge homomorphism $$E_2 ^{q,0}=\Ext ^q _{R\G} (\underline R , M^?) \to H^q (G, M)\cong \Ext^q _{R\G} (R_0 , M^?)$$ is the same as the map induced by the $R\G$-homomorphism $R_0 \to \underline R$ which is defined as the identity map at the trivial subgroup and the zero map at every other subgroup.
\end{corollary}

\begin{proof} We already showed above that the edge homomorphism is the same as the map $\varphi$ which is a map $\Ext ^q _{R\G} (\underline R, M^?) \to \Ext ^q _{R\G} (R_0, M)$ under natural identifications. In the definition of $\varphi$ we take an $\cF$-projective resolution $Q_*$ and a projective resolution $P_*$ of $R$ and define $\varphi$ as the map induced by a chain map $f_*:P_* \to Q_*$. Note that we can consider $P_*$ also as a projective resolution of $R_0$ as an $R\G$-module and we can take $Q_*$ as $S_*(1)$ for some projective 
resolution $S_*$ of $\underline R$ as an $R\G$ module. So, the chain map $f_*$ can be taken as $g_* (1)$ for a chain map $g_* : P_* \to S_*$ of $R\G$-modules which cover the map $R_0 \to \underline R$. So the proof follows from this observation. 
\end{proof}

The spectral sequence given in Theorem \ref{thm:main 3; spect. seq.} has some interesting connections to essential cohomology which is defined as follows: Let $G$ be a finite group and $\cF$ denote the family of all proper subgroups (including the trivial group but not the group $G$ itself). Then the kernel of the edge homomorphism  $$\prod _{H \in \cF} \Res ^G _H : H^* (G, H) \longrightarrow \lim _{\underset{H\in \cF} \longleftarrow } H^* (H, M)$$
is called the {\it essential cohomology} of $G$ and it is denoted by ${\mathcal Ess} ^*(G)$. We see that the essential cohomology classes are exactly the cohomology classes coming from $E_{\infty} ^{p,q}$ with $p>0$. Essential cohomology classes coming from the vertical line $E_{\infty} ^{p,0}$ with $p>0$ are called relative essential cohomology classes and the subring generated by these cohomology classes is called the {\it relative essential cohomology}. Note that relative essential cohomology classes are the essential cohomology classes can be described as extension classes of $\cF$-split extensions. It is interesting to ask how much of the essential cohomology comes from relative essential cohomology classes. This was a question that was raised in \cite{yalcin}. The spectral sequence given in Theorem \ref{thm:main 1; computational tool} can be used to study these types of questions. We only discuss a simple case here.

\begin{example}\label{ex:SpectSeqCalc} Let $G=\bbZ/2 \times \bbZ /2$ and $\cF=\{ 1, H_1, H_2, H_3\}$ be the family of all proper subgroups. Let $R=\bbF _2$. Then, the spectral sequence of Theorem \ref{thm:main 3; spect. seq.} with $M$ equal to the trivial module $R$ has the values $$E_2 ^{p,q}=\Ext ^p _{R\G} (\underline R, H^q (?, R)) \cong \oplus _3 R$$ for all $q>0$ and $p\geq 0$. At $q=0$, the dimensions of $E_2 ^{p,q}$ are given by the sequence $(1,0,1,3,5,7,\dots)$ by the computation in Section \ref{sect:computation}. We claim that in this spectral sequence the horizontal edge homomorphism is the zero map, i.e., all the relative group cohomology on the horizontal line dies at some page of the spectral sequence. To see this, first observe that by Corollary \ref{cor:edgehom}, the horizontal edge homomorphism is given by the map $\varphi : \Ext^q _{R\G} (\underline R, \underline R) \to \Ext ^q _{R\G} (R_0 , \underline R)$. Note that for some $i\in \{1,2,3\}$, the map  $R_0\to \underline R$ can be written as a composition $$R_0 \maprt{\gamma_i} R_{H_i} \maprt{\tau _i} \underline R$$ where $\gamma_i$ is the map we defined in Section \ref{sect:computation} and $\tau_i : R_{H_i}\to \underline{R}$ is the $R\G$-module homomorphism which is defined as the identity map at subgroups $H_i$ and $1$ and the zero map at other subgroups. From this we can see that 
the map $\varphi$  factors through the map 
$\tau_i^*:\Ext^q _{R\G} (\underline R, \underline R) \to \Ext ^q _{R\G} (R_{H_i} , \underline R)$.  In the computation in Section \ref{sect:computation}, we have seen that the map
$$\pi ^*: \Ext^q _{R\G} (\underline R, \underline R) \to \Ext ^q _{R\G} (\oplus R_{H_i} , \underline R)$$ is the zero map for $q\geq 1$. So, $\tau _i ^*$ is also the zero map for all $i\in \{1,2,3\}$. This gives that the horizontal edge homomorphism is zero. This shows in particular that for this group the relative essential cohomology is zero.
\end{example}

{\it Acknowledgements:} The material in the first two sections of the paper is part of the first author's Ph.D. thesis. The first author thanks her thesis advisor Ian Hambleton for his constant support during her Ph.D. Both authors thank Ian Hambleton for his support which made it possible for the authors to meet at McMaster University where the computations appearing in Section \ref{sect:computation} were done. The material appearing in the last two sections is derived from the second author's discussions with Peter Symonds during one of his visits to Bilkent University. The second author thanks T\" UB\. ITAK for making this visit possible and Peter Symonds for introducing him to the connections between higher limits and relative group cohomology.

\providecommand{\bysame}{\leavevmode\hbox to3em{\hrulefill}\thinspace}
\providecommand{\MR}{\relax\ifhmode\unskip\space\fi MR }
\providecommand{\MRhref}[2]{%
  \href{http://www.ams.org/mathscinet-getitem?mr=#1}{#2}}
\providecommand{\href}[2]{#2}

\end{document}